\documentclass[11pt,a4paper]{amsart}
\usepackage{amsfonts,amsmath,amssymb,amsthm,mathtools}
\usepackage[noadjust]{cite}
\usepackage[english]{babel}

\numberwithin{equation}{section}

\DeclareMathOperator{\Ran}{Ran}
\DeclareMathOperator{\spec}{spec}
\DeclareMathOperator{\dist}{dist}
\DeclareMathOperator{\conv}{conv}

\DeclareSymbolFont{SY}{U}{psy}{m}{n}
\DeclareMathSymbol{\emptyset}{\mathord}{SY}{'306}

\DeclarePairedDelimiter{\abs}{|}{|}
\DeclarePairedDelimiter{\norm}{\lVert}{\rVert}

\newcommand{\dd}{\mathrm d}
\newcommand{\eps}{\varepsilon}

\newcommand{\N}{\mathbb{N}}
\newcommand{\R}{\mathbb{R}}
\newcommand{\EE}{\mathsf{E}}

\newcommand{\cH}{{\mathcal H}}
\newcommand{\cO}{{\mathcal O}}

\newtheorem{theorem}{Theorem}[section]{\bf}{\it}
\newtheorem{lemma}[theorem]{Lemma}{\bf}{\it}
\newtheorem{corollary}[theorem]{Corollary}{\bf}{\it}
\newtheorem{proposition}[theorem]{Proposition}{\bf}{\it}
\newtheorem{remark}[theorem]{Remark}{\it}{\rm}

\title[The subspace perturbation problem]{Semidefinite perturbations in the subspace perturbation problem}
\date{}

\author[A.~Seelmann]{Albrecht Seelmann}
\address{A.~Seelmann,
Fakult\"at f\"ur Mathematik, Technische Univer\-si\-t\"at Dortmund,
D-44221 Dortmund, Germany}
\email{albrecht.seelmann@mathematik.tu-dortmund.de}

\begin{document}

\begin{abstract}
 The variation of spectral subspaces for linear self-adjoint operators under an additive bounded semidefinite perturbation is
 considered. A variant of the Davis-Kahan $\sin2\Theta$ theorem from [SIAM J.~Numer.~Anal.~\textbf{7} (1970), 1--46] adapted to
 this situation is proved. Under a certain additional geometric assumption on the separation of the spectrum of the unperturbed
 operator, this leads to a sharp estimate on the norm of the difference of the spectral projections associated with isolated
 components of the spectrum of the perturbed and unperturbed operators, respectively. Without this additional geometric assumption
 on the isolated components of the spectrum of the unperturbed operator, a corresponding estimate is obtained by transferring the
 optimization approach for general perturbations in [J.~Anal.~Math.~\textbf{135} (2018), 313--343] to the present situation.
\end{abstract}

\maketitle

\section{Introduction}\label{sec:intro}

The~\emph{subspace perturbation problem} deals with the variation of spectral subspaces for a self-adjoint operator under an
additive perturbation and has previously been discussed in several recent works such as~\cite{AM13,KMM03,KMM07,MS15,Seel14,Seel16,
Seel18}. In the present work, we continue these considerations and study the problem in the particular case of semidefinite
perturbations.

Let $A$ be a self-adjoint, not necessarily bounded, operator on a separable Hilbert space $\cH$. Moreover, let $V$ be a bounded
self-adjoint operator on $\cH$ which is non-negative, that is, $V\ge 0$. The consideration of non-positive perturbations $V$, that
is, $V\le0$, is analogous and can, in view of the identity $-(A+V)=-A+(-V)$, also be reduced to the case of non-negative
perturbations.

It is well known that a semidefinite perturbation moves the spectrum only in one direction. More precisely, if $(a,b)\subset\R$,
$a<b$, is an interval in the resolvent set of $A$ and $V\ge 0$ satisfies $\norm{V}<b-a$, then the interval $(a+\norm{V},b)$ is
contained in the resolvent set of the perturbed operator $A+V$, see, e.g.,~\cite[Theorem~3.2]{Ves08} and also
Proposition~\ref{prop:semidefinite} below. As a consequence, if the spectrum of $A$ is separated into two disjoint components,
that is,
\begin{equation}\label{eq:intro:spec}
 \spec(A) = \sigma \cup \Sigma \quad\text{ with }\quad d:=\dist(\sigma,\Sigma)>0,
\end{equation}
and if the norm of the perturbation satisfies
\begin{equation}\label{eq:normBound}
 \norm{V}<d,
\end{equation}
then the spectrum of the perturbed operator $A+V$ is likewise separated into two disjoint components,
\begin{equation}\label{eq:intro:specPert}
 \spec(A+V) = \omega \cup \Omega \quad\text{ with }\quad \dist(\omega,\Omega)\ge d-\norm{V} > 0,
\end{equation}
where $\omega$ and $\Omega$ are contained in certain ``right-side'' neighbourhoods of $\sigma$ and $\Sigma$, respectively. Namely,
\begin{equation}\label{eq:intro:defomega}
 \omega = \spec(A+V)\cap\bigl(\sigma+[0,\norm{V}]\bigr)
\end{equation}
and analogously for $\Omega$ (with $\sigma$ replaced by $\Sigma$); here we used the notation
$\sigma+[0,\norm{V}]:=\{\lambda+t \mid \lambda\in\sigma\,,\ 0\le t\le\norm{V}\}$. Clearly,
the~\emph{gap non-closing condition}~\eqref{eq:normBound} is sharp.

The variation of the spectral subspaces associated with the components of the spectrum is studied in terms of the corresponding
spectral projections $\EE_A(\sigma)$ and $\EE_{A+V}(\omega)$, where $\EE_A$ and $\EE_{A+V}$ denote the projection-valued spectral
measures for the self-adjoint operators $A$ and $A+V$, respectively. Here, the quantity
\begin{equation}\label{eq:defmaxangle}
 \theta := \arcsin\bigl(\norm{\EE_A(\sigma)-\EE_{A+V}(\omega)}\bigr)
\end{equation}
is called the~\emph{maximal angle} between the two spectral subspaces $\Ran\EE_A(\sigma)$ and $\Ran\EE_{A+V}(\omega)$. Recall that
always $\norm{\EE_A(\sigma)-\EE_{A+V}(\omega)}\le1$, so that $\theta$ in \eqref{eq:defmaxangle} is well defined. Moreover, if
$\norm{\EE_A(\sigma)-\EE_{A+V}(\omega)}<1$, that is, if $\theta<\pi/2$, then the projections $\EE_A(\sigma)$ and
$\EE_{A+V}(\omega)$ are unitarily equivalent, see, e.g., \cite[Theorem~I.6.32]{Kato66}. In this case, the perturbed subspace
$\Ran\EE_{A+V}(\omega)$ can be understood as a rotation of the unperturbed subspace $\Ran\EE_A(\sigma)$ and the maximal angle
$\theta$ serves as a measure for this rotation.

In this context, it is a natural question whether the condition \eqref{eq:normBound} is sufficient to ensure that $\theta<\pi/2$.
More specifically, we ask for the best possible constant $c_\text{opt-sem}\in(0,1]$ such that
\begin{equation}\label{eq:acute}
 \theta < \frac{\pi}{2} \quad\text{ whenever }\quad \norm{V}<c_\text{opt-sem}\cdot d.
\end{equation}

The analogous problem has previously been discussed for off-diagonal perturbations (see~\cite{KMM07,MS15,Seel16} and the references
therein) and for general, not necessarily semidefinite or off-diagonal, perturbations, see~\cite{KMM03,MS15,AM13,Seel18}. For the
latter, the (sharp) gap non-closing condition reads $\norm{V}<d/2$, in which case instead of $\omega$ in~\eqref{eq:intro:defomega}
the component of $\spec(A+V)$ in $\cO_{d/2}(\sigma)$, the open $d/2$-neighbourhood of $\sigma$, is considered (and similarily for
$\Omega$). Here, one is interested in the best possible constant $c_\text{opt}\in(0,1/2]$ analogous to $c_\text{opt-sem}$
in~\eqref{eq:acute}. Under a certain additional geometric assumption on the spectrum of $A$ it is known that $c_\text{opt}=1/2$ and
a corresponding (sharp) estimate on the maximal angle reads
\[
 \arcsin\bigl(\norm{\EE_A(\sigma)-\EE_{A+V}\bigl(\cO_{d/2}(\sigma)\bigr)}\bigr) \le
 \frac{1}{2}\arcsin\Bigl(2\frac{\norm{V}}{d}\Bigr) < \frac{\pi}{4}
\]
if $\norm{V}<d/2$, see, e.g.,~\cite[Remark 2.9]{Seel14}. Astonishingly, in the present situation of semidefinite perturbations and
with $\cO_{d/2}(\sigma)$ replaced by $\omega$ as in~\eqref{eq:intro:defomega}, the term $\norm{V}$ in this estimate can formally be
replaced by $\norm{V}/2$, thereby allowing the whole scope of semidefinite perturbations satisfying~\eqref{eq:normBound}. The
precise statement is as follows:

\begin{theorem}\label{thm:favGeo}
 Let $A$ be a self-adjoint operator on a separable Hilbert space $\cH$ such that the spectrum of $A$ is separated as
 in~\eqref{eq:intro:spec}. Let $V$ be a non-negative bounded self-adjoint operator on $\cH$ with $\norm{V}<d$, and choose the
 spectral component $\omega$ of $\spec(A+V)$ as in~\eqref{eq:intro:defomega}.
 
 If, in addition, the convex hull of one of the components $\sigma$ and $\Sigma$ is disjoint from the other component, that is,
 $\conv(\sigma)\cap\Sigma=\emptyset$ or vice versa, then
 \begin{equation}\label{eq:favGeo}
  \arcsin\bigl(\norm{\EE_A(\sigma)-\EE_{A+V}(\omega)}\bigr) \le \frac{1}{2}\arcsin\Bigl(\frac{\norm{V}}{d}\Bigr) < \frac{\pi}{4},
 \end{equation}
 and this estimate is sharp.
\end{theorem}

As a consequence, under the additional geometric assumption on the spectrum of $A$ in Theorem~\ref{thm:favGeo}, namely
$\conv(\sigma)\cap\Sigma=\emptyset$ or $\sigma\cap\conv(\Sigma)=\emptyset$, one has $c_\text{opt-sem}=1$.
However, without any additional hypotheses on the spectrum of $A$, that is, under the sole assumption~\eqref{eq:intro:spec}, the
value of $c_\text{opt-sem}$ is still unknown. In the case where $A$ is assumed to be bounded and $V$ has rank one, it has recently
been shown in~\cite[Theorem~2.10]{Gebert18} that
\[
 \norm{\EE_A(\sigma)-\EE_{A+V}(\omega)} \le \frac{\norm{V}}{d}<1 \quad\text{ if }\quad \norm{V}<d,
\]
yielding $c_\text{opt-sem}=1$ in this very particular situation. However, it is also acknowledged there that the corresponding
proof only works for rank one perturbations. For general semidefinite perturbations, only lower bounds on $c_\text{opt-sem}$ can be
given so far. This is quite similar to the case of general, not necessarily semidefinite perturbations mentioned above. There, the
currently best known result~\cite[Theorem~1]{Seel18} states that
\[
 c_\text{opt}\ge c_\text{crit} = \frac{1}{2} - \frac{1}{2}\Bigl(1-\frac{\sqrt{3}}{\pi}\Bigr)^3 = 0{.}4548399\ldots
\]
and
\[
 \arcsin\bigl(\norm{\EE_A(\sigma)-\EE_{A+V}\bigl(\cO_{d/2}(\sigma)\bigr)}\bigr) \le N\Bigl(\frac{\norm{V}}{d}\Bigr) < \frac{\pi}{2}
\]
if $\norm{V} < c_\text{crit}\cdot d$, where $N\colon[0,c_\text{crit}]\to[0,\pi/2]$ is given by
\begin{equation}\label{eq:boundFunc}
 N(x) =
 \begin{cases}
  \frac{1}{2}\arcsin(\pi x) & \text{ for }\quad 0\le x\le \frac{4}{\pi^2+4},\\[0.15cm]
  \arcsin\Bigl(\sqrt{\frac{2\pi^2x-4}{\pi^2-4}}\,\Bigr) & \text{ for }\quad \frac{4}{\pi^2+4} < x < 4\frac{\pi^2-2}
    {\pi^4},\\[0.15cm]
  \arcsin\bigl(\frac{\pi}{2}(1-\sqrt{1-2x}\,)\bigr) & \text{ for }\quad 4\frac{\pi^2-2}{\pi^4} \le x \le \kappa,\\[0.15cm]
  \frac{3}{2}\arcsin\bigl(\frac{\pi}{2}(1-\sqrt[\leftroot{4}3]{1-2x}\,)\bigr) & \text{ for }\quad
    \kappa < x \le c_\text{crit}.
 \end{cases}
\end{equation}
Here, $\kappa\in\bigl(4\frac{\pi^2-2}{\pi^4},2\frac{\pi-1}{\pi^2}\bigr)$ is the unique solution to the equation
\[
 \arcsin\Bigl(\frac{\pi}{2}\bigl(1-\sqrt{1-2\kappa}\,\bigr)\Bigr)=\frac{3}{2}\arcsin\Bigl(\frac{\pi}{2}
 \bigl(1-\sqrt[\leftroot{4}3]{1-2\kappa}\,\bigr)\Bigr)
\]
in the interval $\bigl(0,2\frac{\pi-1}{\pi^2}\bigr]$.

As in the case of Theorem~\ref{thm:favGeo}, it turns out that this result can just as well be adapted to the present situation of
semidefinite perturbations by formally replacing $\cO_{d/2}(\sigma)$ and $\norm{V}$ by $\omega$ and $\norm{V}/2$, respectively. This
leads to the conclusion that $c_\text{opt-sem}\ge2c_\text{crit}$, the second principal result in this work:

\begin{theorem}[cf.~{\cite[Theorem~1]{Seel18}}]\label{thm:generic}
 Let $A$ be a self-adjoint operator on a separable Hilbert space $\cH$ such that the spectrum of $A$ is separated as
 in~\eqref{eq:intro:spec}, and let $V$ and $\omega$ be as in Theorem~\ref{thm:favGeo}. If, in addition, $V$ satisfies
 \[
  \norm{V} < c_\mathrm{crit{\text-}sem}\cdot d
 \]
 with
 \[
  c_\mathrm{crit{\text-}sem} = 1 - \Bigl(1-\frac{\sqrt{3}}{\pi}\Bigr)^3 = 0{.}9096799\dots,
 \] 
 then
 \begin{equation}\label{eq:boundGeneric}
  \arcsin\bigl(\norm{\EE_A({\sigma})-\EE_{A+V}(\omega))}\bigr) \le N\Bigl(\frac{\norm{V}}{2d}\Bigr)
  < \frac{\pi}{2},
 \end{equation}
 where $N$ is given by~\eqref{eq:boundFunc}.
\end{theorem}
A more detailed discussion on the function $N$ can be found in~\cite{Seel18}.

The proofs of Theorems~\ref{thm:favGeo} and~\ref{thm:generic} rely on the following variant of the Davis-Kahan $\sin2\Theta$
theorem for semidefinite perturbations:
\begin{equation}\label{eq:introSin2Theta}
 \norm{\sin2\Theta} \le \frac{\pi}{2}\frac{\norm{V}}{d},
\end{equation}
where $\Theta=\arcsin\abs{\EE_A(\sigma)-\EE_{A+V}(\omega)}$ with $\norm{\Theta}=\theta$ is the operator angle associated with
$\EE_A(\sigma)$ and $\EE_{A+V}(\omega)$; the constant $\pi/2$ here can be replaced by $1$ if $\conv(\sigma)\cap\Sigma=\emptyset$ or
$\sigma\cap\conv(\Sigma)=\emptyset$, see Proposition~\ref{prop:sin2Theta} below. The estimate~\eqref{eq:introSin2Theta} differs
from the corresponding variant for general perturbations in~\cite{Seel14} (cf.~\cite{DK70}) by the lack of a factor $2$ on its
right-hand side, which is the result of a suitable adaptation to the proof presented in~\cite{Seel14}.

The paper is organized as follows:

Section~\ref{sec:semidefinite} is devoted to preliminaries regarding the perturbation of the spectrum by semidefinite
perturbations.

In Section~\ref{sec:sin2Theta}, a variant of the Davis-Kahan $\sin2\Theta$ theorem for semidefinite perturbations is proved and
Theorem~\ref{thm:favGeo} is deduced.

The proof of Theorem~\ref{thm:generic} is presented in Section~\ref{sec:generic}.

Finally, an alternative, direct proof for a variant of the $\sin2\theta$ estimate, related to~\eqref{eq:introSin2Theta} by
$\sin2\theta\le \norm{\sin2\Theta}$, is discussed in Appendix~\ref{app:sin2theta}. This proof is the result of an adaptation to the
corresponding direct proof of the generic $\sin2\theta$ estimate from~\cite[Proposition~3.3]{Seel14}. The key ingredient in this
adaptation, Lemma~\ref{lem:semidefinite} below, may also be of independent interest.

\section{Semidefinite perturbations}

\subsection{Perturbation of the spectrum}\label{sec:semidefinite}

The following result is extracted from the more general statement~\cite[Theorem~3.2]{Ves08}; cf.~also~\cite[Eq.~(9.4.4)]{BS87}.

\begin{proposition}\label{prop:semidefinite}
 Let $A$ be a self-adjoint operator on the Hilbert space $\cH$ such that its resolvent set contains a finite interval
 $(a,b)\subset\R$, $a<b$. Moreover, let $V$ be a non-negative (resp.~non-positive) bounded self-adjoint operator on $\cH$.
 
 If $\norm{V}<b-a$, then the interval $(a+\norm{V},b)$ (resp.~$(a,b-\norm{V})$) belongs to the resolvent set of the perturbed
 operator $A+V$.

 \begin{proof}
  For the sake of completeness, we reproduce the proof.

  Let $\norm{V}<b-a$ and assume that $V$ is non-negative. The case where $V$ is non-positive can be reduced to this case by
  considering $-(A+V)=-A-V$.
  
  Denote $\cH_-:=\Ran\EE_A\bigl((-\infty,a]\bigr)$ and $\cH_+:=\Ran\EE_A\bigl([b,\infty)\bigr)$, and denote by
  $A_\pm:=A|_{\cH_\pm}$ the parts of $A$ associated with $\cH_\pm$. Decompose the perturbation $V$ as
  \[
   V = V_\text{diag}+V_\text{off},
  \]
  where $V_\text{diag}=V_-\oplus V_+$ is the diagonal part of $V$ and $V_\text{off}$ is the off-diagonal part of $V$ with respect
  to the orthogonal decomposition $\cH=\cH_-\oplus\cH_+$.

  Since $V$ is non-negative, the diagonal part $V_\text{diag}$ of $V$ is non-negative as well, that is, $V_\pm\ge 0$. Thus, taking
  into account that $a+\norm{V}<b$, one has
  \[
   A_-+V_- \le a+\norm{V} < b \le A_++V_+.
  \]
  In particular, the interval $(a+\norm{V},b)$ belongs to the resolvent set of the operator
  $A+V_\text{diag}=(A_-+V_-)\oplus(A_++V_+)$, and the subspaces $\cH_-$ and $\cH_+$ are the spectral subspaces for
  $A+V_\text{diag}$ associated with the sets $(-\infty,a+\norm{V}]$ and $[b,\infty)$, respectively. Now, by~\cite[Theorem~2.1]{AL95}
  (see also~\cite[Theorem~8.1]{DK70}), the gap $(a+\norm{V},b)$ in the spectrum is preserved under the off-diagonal perturbation
  $V_\text{off}$, that is, the interval $(a+\norm{V},b)$ also belongs to the resolvent set of $A+V=(A+V_\text{diag})+V_\text{off}$.
 \end{proof}%
\end{proposition}

As a direct consequence of Proposition~\ref{prop:semidefinite}, semidefinite perturbations move the spectrum only in one direction.
More precisely, using the notation $\Delta+[0,\norm{V}]:=\{\lambda+t\mid \lambda\in\Delta\,,\ 0\le t\le\norm{V}\}$ for a Borel set
$\Delta\subset\R$, we have the following corollary to Proposition~\ref{prop:semidefinite}.

\begin{corollary}\label{cor:semidefinite}
 Let $A$ be a self-adjoint operator on the Hilbert space $\cH$, and let $V$ be a non-negative bounded self-adjoint operator on
 $\cH$. Then,
 \[
  \spec(A+V) \subset \spec(A) + [0,\norm{V}].
 \]
 
 \begin{proof}
  Let $\lambda\in\spec(A+V)$ be arbitrary. We have to show that $\spec(A)$ and the interval $[\lambda-\norm{V},\lambda]$ intersect.
  
  Assume the contrary, that is, $[\lambda-\norm{V},\lambda]\subset\rho(A)$. Since $\rho(A)$ is open, there is $\eps>0$ such that
  $(\lambda-\norm{V}-\eps,\lambda+\eps)\subset\rho(A)$. Proposition~\ref{prop:semidefinite} then implies that
  $(\lambda-\eps,\lambda+\eps)\subset\rho(A+V)$, which contradicts $\lambda\in\spec(A+V)$. Hence,
  $\spec(A)\cap[\lambda-\norm{V},\lambda]\neq\emptyset$, and the proof is complete.
 \end{proof}%
\end{corollary}

In the situation of Theorems~\ref{thm:favGeo} and~\ref{thm:generic}, it is easy to see from Corollary~\ref{cor:semidefinite} that
the spectrum of the perturbed operator $A+V$ is separated as in~\eqref{eq:intro:specPert} and~\eqref{eq:intro:defomega}, where
$\omega$ and $\Omega$ are non-empty and contained in one-sided neighbourhoods of $\sigma$ and $\Sigma$, respectively.

In the same way, for each $t\in[0,1]$ the spectrum of the operator $A+tV$ is separated into two disjoint components $\omega_t$ and
$\Omega_t$ defined analogously to $\omega$ and $\Omega$ above, respectively, that is,
\begin{equation}\label{eq:defomega_t}
 \omega_t = \spec(A+tV) \cap \bigl(\sigma+[0,t\norm{V}]\bigr)
\end{equation}
and
\begin{equation}\label{eq:defOmega_t}
 \Omega_t = \spec(A+tV) \cap \bigl(\Sigma+[0,t\norm{V}]\bigr).
\end{equation}

In this context, we need the following result for future reference.

\begin{lemma}[{cf.~\cite[Theorem 3.5]{AM13}}]\label{lem:cont}
 Let $A$ be as in Theorem~\ref{thm:generic}, and let $V$ be a non-negative bounded self-adjoint operator on $\cH$ satisfying
 $\norm{V}<d$. For $t\in[0,1]$ consider the spectral component $\omega_t\subset\spec(A+tV)$ as in~\eqref{eq:defomega_t}. Then, the
 path $[0,1]\ni t\mapsto\EE_{A+tV}(\omega_t)$ of spectral projections is continuous in norm.
  
 \begin{proof}
  It is easy to see that for $0\le s\le t\le 1$ the spectral components~\eqref{eq:defomega_t} and~\eqref{eq:defOmega_t} satisfy
  \[
   \dist(\omega_s,\Omega_t)\ge d-t\norm{V} \quad\text{ and }\quad \dist(\Omega_s,\omega_t)\ge d-t\norm{V}.
  \]
  Taking into account that $A+tV=(A+sV)+(t-s)V$, the symmetric $\sin\Theta$ theorem from~\cite[Proposition~2.3]{Seel14} (see also,
  e.g., the proof of~\cite[Theorem~3.5]{AM13}) then implies that
  \[
   \norm{\EE_{A+sV}(\omega_s)-\EE_{A+tV}(\omega_t)} \le \frac{\pi}{2}\frac{\abs{t-s}\norm{V}}{d-t\norm{V}}
   \quad\text{ for }\quad 0\le s\le t\le 1,
  \]
  which immediately proves the claim.
 \end{proof}%
\end{lemma}

\subsection{The $\sin2\Theta$ theorem for semidefinite perturbations}\label{sec:sin2Theta}

The following result provides a variant of the Davis-Kahan $\sin2\Theta$ theorem (see, e.g.,~\cite{DK70} and~\cite{Seel14}) for
semidefinite perturbations. This is the core of the proofs of both Theorems~\ref{thm:favGeo} and~\ref{thm:generic}.

\begin{proposition}[cf.~{\cite[Theorem~1]{Seel14}}]\label{prop:sin2Theta}
 Let $A$ be as in Theorem~\ref{thm:generic}. Moreover, let $V$ be a bounded non-negative operator on $\cH$ and $Q$ be an orthogonal
 projection in $\cH$ onto a reducing subspace for $A+V$. Then, the operator angle $\Theta=\arcsin\abs{\EE_A(\sigma)-Q}$ associated
 with $\EE_A(\sigma)$ and $Q$ satisfies
 \begin{equation}\label{eq:sin2Theta}
  \norm{\sin2\Theta} \le \frac{\pi}{2}\frac{\norm{V}}{d}.
 \end{equation}
 If, in addition, $\conv(\sigma)\cap\Sigma=\emptyset$ or $\sigma\cap\conv(\Sigma)=\emptyset$, then the constant $\pi/2$
 in~\eqref{eq:sin2Theta} can be replaced by $1$.

 \begin{proof}
  Recall from the proof of~\cite[Theorem~1]{Seel14} that
  \[
   \norm{\sin2\Theta} \le \frac{\pi}{2}\frac{\norm{V-KVK}}{d},
  \]
  where $K=Q-Q^\perp=2Q-I_\cH$ is self-adjoint and unitary. Also recall that the constant $\pi/2$ in this estimate can be replaced
  by $1$ if $\conv(\sigma)\cap\Sigma=\emptyset$ or $\sigma\cap\conv(\Sigma)=\emptyset$, see, e.g.,~\cite[Remark~2.5]{Seel14}. It
  only remains to show that $\norm{V-KVK}\le\norm{V}$.

  Indeed, since $V$ is non-negative, the operator $KVK$ is non-negative as well and, thus,
  \[
   -KVK \le V-KVK \le V.
  \]
  Hence, $\norm{V-KVK}\le\max\{\norm{V},\norm{KVK}\}=\norm{V}$, where we have taken into account that $K$ is unitary. This
  completes the proof.
 \end{proof}%
\end{proposition}

We are now able to prove Theorem~\ref{thm:favGeo} by taking $Q=\EE_{A+V}(\omega)$ in Proposition~\ref{prop:sin2Theta}:

\begin{proof}[Proof of Theorem~\ref{thm:favGeo}]
 Recall that the maximal angle $\theta$ in~\eqref{eq:defmaxangle} agrees with the norm of the operator angle
 $\Theta=\arcsin\abs{\EE_A(\sigma)-\EE_{A+V}(\omega)}$ associated with $\EE_A(\sigma)$ and $\EE_{A+V}(\omega)$, cf.,
 e.g.,~\cite[Eq.~(2.6)]{Seel14}. Thus, it follows from Proposition~\ref{prop:sin2Theta} that
 \begin{equation}\label{eq:sin2tT}
  \sin2\theta \le \norm{\sin2\Theta} \le \frac{\norm{V}}{d}.
 \end{equation}
 Moreover, combining~\eqref{eq:sin2tT} and Lemma~\ref{lem:cont}, the same continuity argument as in the proof
 of~\cite[Lemma~2.7]{Seel14} shows that $\theta\le\pi/4$. Inequality~\eqref{eq:sin2tT} then agrees with estimate~\eqref{eq:favGeo}.

 The sharpness of estimate~\eqref{eq:favGeo} can be seen from the following example of $2\times2$ matrices
 (cf.~\cite[Remark~2.9]{Seel14}): For arbitrary $0\le v<1$ consider
 \[
  A:=\begin{pmatrix} -\frac{1}{2} & 0\\ 0 & \frac{1}{2} \end{pmatrix} \quad\text{and }\quad
  V:=\begin{pmatrix} \frac{v(v+1)}{2} & \frac{v\sqrt{1-v^2}}{2}\\[0.1cm] \frac{v\sqrt{1-v^2}}{2} & \frac{v(1-v)}{2} \end{pmatrix}
 \]
 with $\sigma:=\{-1/2\}$, $\Sigma:=\{1/2\}$, and $d:=\dist(\sigma,\Sigma)=1$.

 It is easy to verify that $\spec(V)=\{0,v\}$, hence $V\ge0$ and $\norm{V}=v$, and that the spectrum of $A+V$ is given by
 $\spec(A+V)=\bigl\{(v\pm\sqrt{1-v^2})/2\bigr\}$.

 Denote $\omega:=\bigl\{(v-\sqrt{1-v^2})/2\bigr\}\subset[-1/2,-1/2+v]$ and $\theta:=\arcsin(v)/2$. Using the identities
 \[
  \frac{1-\sqrt{1-v^2}}{v} = \tan\theta = \frac{v}{1+\sqrt{1-v^2}} \quad\text{ for }\quad 0<v<1,
 \]
 it is then straightforward to show that
 \[
  (A+V)\begin{pmatrix} \cos\theta\\ -\sin\theta \end{pmatrix} =
  \frac{v-\sqrt{1-v^2}}{2}\begin{pmatrix} \cos\theta\\ -\sin\theta\end{pmatrix},
 \]
 and, therefore,
 \[
  \arcsin\bigl(\norm{\EE_A(\sigma)-\EE_{A+V}(\omega)}\bigr) = \theta = \frac{1}{2}\arcsin\Bigl(\frac{\norm{V}}{d}\Bigr).
 \]
 Hence, estimate~\eqref{eq:favGeo} is sharp, which completes the proof.
\end{proof}%

Analogously to the proof of Theorem~\ref{thm:favGeo}, we obtain the following corollary to Proposition~\ref{prop:sin2Theta} in the
situation where no additional assumptions on the spectrum of $A$ are imposed. This result plays a crucial role in the proof of
Theorem~\ref{thm:generic}, see Section~\ref{sec:generic} below.

\begin{corollary}[cf.~{\cite[Corollary~2]{Seel14}}]\label{cor:generic}
 In the situation of Theorem~\ref{thm:generic} one has
 \[
  \theta \le \frac{1}{2}\arcsin\Bigl(\frac{\pi}{2}\frac{\norm{V}}{d}\Bigr) \le \frac{\pi}{4}
  \quad\text{ whenever }\quad \norm{V}\le\frac{2d}{\pi}.
 \]
\end{corollary}

It is interesting to note that also the alternative, direct proof of the $\sin2\theta$ estimate from~\cite[Proposition~3.3]{Seel14},
which is related to Proposition~\ref{prop:sin2Theta} by $\sin2\theta\le\norm{\sin2\Theta}$ (cf.~equation~\eqref{eq:sin2tT} above),
can be adapted to the case of semidefinite perturbations. This is discussed in Appendix~\ref{app:sin2theta} below.

\subsection{Proof of Theorem \ref{thm:generic}}\label{sec:generic}
For $t\in[0,1]$ let $P_t:=\EE_{A+tV}(\omega_t)$ denote the spectral projection for $A+tV$ associated with the spectral component
$\omega_t$ in~\eqref{eq:defomega_t}. Clearly, one has $P_0=\EE_A(\sigma)$ and $P_1=\EE_{A+V}(\omega)$.

Let $0=t_0\le\dots\le t_n=1$, $n\in\N$, be a finite partition of the interval $[0,1]$. As in~\cite{AM13},~\cite{Seel18},
and~\cite{Seel16}, the triangle inequality for the maximal angle (see, e.g.,~\cite[Corollary~4]{Brown93}) yields
\begin{equation}\label{eq:projTriangle}
 \arcsin\bigl(\norm{\EE_A(\sigma)-\EE_{A+V}(\omega)}\bigr) \le \sum_{j=0}^{n-1} \arcsin\bigl(\norm{P_{t_j}-P_{t_{j+1}}}\bigr).
\end{equation}
Moreover, one has $\dist(\omega_{t_j},\Omega_{t_j})\ge d-t_j\norm{V}$ and, therefore,
\[
 \frac{(t_{j+1}-t_j)\norm{V}}{\dist(\omega_{t_j},\Omega_{t_j})} \le \frac{(t_{j+1}-t_j)\norm{V}}{d-t_j\norm{V}} =: \lambda_j < 1,
 \quad j=0,\dots,n-1.
\]
Hence, considering $A+t_{j+1}V=(A+t_jV)+(t_{j+1}-t_j)V$ as a perturbation of $A+t_jV$ with $(t_{j+1}-t_j)V\ge0$ and taking into
account that
\[
 \omega_{t_{j+1}}=\spec(A+t_{j+1}V)\cap \bigl(\omega_{t_j}+[0,(t_{j+1}-t_j)\norm{V}]\bigr),
\]
it follows from Corollary~\ref{cor:generic} that
\begin{equation}\label{eq:localSin2theta}
 \arcsin\bigl(\norm{P_{t_j}-P_{t_{j+1}}}\bigr) \le \frac{1}{2}\arcsin\Bigl(\frac{\pi\lambda_j}{2}\Bigr) \quad\text{ whenever }\quad
 \lambda_j\le\frac{2}{\pi}.
\end{equation}
Combining inequalities~\eqref{eq:projTriangle} and~\eqref{eq:localSin2theta} suggests to estimate the maximal angle between the
subspaces $\Ran\EE_A(\sigma)$ and $\Ran\EE_{A+V}(\omega)$ as
\begin{equation}\label{eq:estimate}
 \arcsin\bigl(\norm{\EE_A(\sigma)-\EE_{A+V}(\omega)}\bigr) \le
 \frac{1}{2}\sum_{j=0}^{n-1} \arcsin\Bigl(\frac{\pi\lambda_j}{2}\Bigr),
\end{equation}
provided that $\lambda_j\le 2/\pi$. The task then is to minimize the right-hand side of~\eqref{eq:estimate} over all corresponding
choices of partitions of the interval $[0,1]$.

In this context, the consideration of partitions of the interval $[0,1]$ with arbitrarily small mesh size allows one to obtain an
analogue to the bounds in~\cite[Theorems~3.2 and~3.3]{MS15} in the current case of semidefinite perturbations:
\begin{remark}[{cf.~\cite[Section 2]{Seel18} and~\cite[Remark~2.1]{Seel16}}]\label{rem:MS}
 If the mesh size of the partition of the interval $[0,1]$ is sufficiently small, then the Riemann sum
 \[
  \sum_{j=0}^{n-1}\lambda_j = \sum_{j=0}^{n-1} \frac{(t_{j+1}-t_j)\norm{V}}{d-t_j\norm{V}}
 \]
 is close to the integral $\int_0^1 \frac{\norm{V}}{d-t\norm{V}}\,\dd t$. Since at the same time each $\lambda_j$ is small and
 $\arcsin(x)/x\to 1$ as $x\to 0$, we conclude from~\eqref{eq:estimate} that
 \[
  \arcsin\bigl(\norm{\EE_A(\sigma)-\EE_{A+V}(\omega)}\bigr) \le \frac{\pi}{4} \int_0^1 \frac{\norm{V}}{d-t\norm{V}}\,\dd t
  = \frac{\pi}{4}\log\frac{d}{d-\norm{V}}.
 \]
 Here, the right-hand side of the latter inequality is strictly less than $\pi/2$ whenever
 $\norm{V}/d < 2\sinh(1)/\exp(1)=0{.}86466\ldots<c_\mathrm{crit-sem}$.
\end{remark}

Clearly, one has $(t_{j+1}-t_j)\norm{V}/d=(1-t_j\norm{V}/d)\lambda_j$ by definition of $\lambda_j$, which can equivalently be
rewritten as
\[
 1-t_{j+1}\frac{\norm{V}}{d} = \Bigl(1-t_j\frac{\norm{V}}{d}\Bigr)(1-\lambda_j),\quad j=0,\dots,n-1.
\]
Since $t_0=0$ and $t_n=1$, this implies that
\begin{equation}\label{eq:partRepr}
 1-\frac{\norm{V}}{d} = \prod_{j=0}^{n-1} (1-\lambda_j).
\end{equation}
The right-hand side of~\eqref{eq:estimate} may therefore equivalently be minimized over all choices of $n\in\N$ and parameters
$\lambda_j\in[0,2/\pi]$, $j=0,\dots,n-1$, satisfying~\eqref{eq:partRepr}. It turns out that this optimization problem is in fact
just the same as the one in~\cite{Seel18}:

\begin{proof}[Proof of Theorem~\ref{thm:generic}]
 Recall from~\cite{Seel18} that the function $N$ in~\eqref{eq:boundFunc} is given by
 \[
  N(x)=\inf\biggl\{ \frac{1}{2}\sum_{j=0}^{n-1} \arcsin(\pi\lambda_j) \biggm| n\in\N,\, 0\le\lambda_j\le\frac{1}{\pi},\,
  \prod_{j=0}^{n-1}(1-2\lambda_j)=1-2x\biggr\}
 \]
 for $0\le x\le c_\text{crit}=\bigl(1-(1-\sqrt{3}/\pi)^3\bigr)/2$. Hence, replacing $2\lambda_j$ with $\lambda_j$, one
 obviously has
 \[
  N\Bigl(\frac{x}{2}\Bigr)=\inf\biggl\{ \frac{1}{2}\sum_{j=0}^{n-1} \arcsin\Bigl(\frac{\pi\lambda_j}{2}\Bigr) \biggm| n\in\N,\,
  0\le\lambda_j\le\frac{2}{\pi},\, \prod_{j=0}^{n-1}(1-\lambda_j)=1-x\biggr\}
 \]
 for $0\le x\le 2c_\text{crit}=c_\text{crit-sem}$. Taking into account~\eqref{eq:estimate} and~\eqref{eq:partRepr}, this
 proves~\eqref{eq:boundGeneric} and, thus, the claim.
\end{proof}%

\appendix

\section{The $\sin2\theta$ estimate}\label{app:sin2theta}
The aim of this section is to show how the direct proof of~\cite[Proposition~3.3]{Seel14} can be adapted to obtain the following
variant of the $\sin2\theta$ estimate for semidefinite perturbations.

\begin{proposition}[cf.~{\cite[Proposition~3.3]{Seel14}}]\label{prop:sin2theta}
 Let $A$, $V$, and $Q$ be as in Proposition~\ref{prop:sin2Theta}. Then,
 \begin{equation}\label{eq:sin2theta}
  \sin2\theta \le \frac{\pi}{2}\frac{\norm{V}}{d},
 \end{equation}
 where $\theta:=\arcsin\bigl(\norm{\EE_A(\sigma)-Q}\bigr)$ is the maximal angle associated with the subspaces $\Ran\EE_A(\sigma)$
 and $\Ran Q$.

 If, in addition, $\conv(\sigma)\cap\Sigma=\emptyset$ or $\sigma\cap\conv(\Sigma)=\emptyset$, then the constant $\pi/2$
 in~\eqref{eq:sin2theta} can be replaced by $1$.
\end{proposition}

The key to obtain this variant from the proof of~\cite[Proposition~3.3]{Seel14} is the following observation for semidefinite
bounded operators, which might also be of independent interest.

\begin{lemma}\label{lem:semidefinite}
 Let $V$ be a non-negative bounded self-adjoint operator on $\cH$ given by the $2\times2$ block operator matrix
 \[
  V = \begin{pmatrix} V_0 & W\\ W^* & V_1 \end{pmatrix}
 \]
 with respect to an orthogonal decomposition $\cH=\cH_0\oplus\cH_1$. Then, one has
 \begin{equation}\label{eq:help}
  2\norm{W} \le \norm{V} \le 2\max\{\norm{V_0},\norm{V_1}\}.
 \end{equation}

 \begin{proof}
  For arbitrary normalized vectors $f\in\cH_0$ and $g\in\cH_1$ define the Hermitian $2\times2$ scalar matrix
  \[
   T_{f,g}:= \begin{pmatrix}
              \langle f,V_0f\rangle & \langle f,Wg\rangle\\
              \langle g,W^*f\rangle & \langle g,V_1g\rangle
             \end{pmatrix}.
  \]
  This matrix is again non-negative  and one has
  \[
   \sup_{\norm{f}=1=\norm{g}} \norm{T_{f,g}} = \norm{V},
  \]
  which can be seen from the identity
  \[
   \begin{pmatrix} \overline{\alpha} & \overline{\beta} \end{pmatrix} T_{f,g} \begin{pmatrix} \alpha\\ \beta \end{pmatrix} =
   \langle x,Vx\rangle
  \]
  for arbitary scalars $\alpha,\beta$ and $x:=\alpha f\oplus \beta g\in\cH$, see, e.g.,~\cite[Theorem~1.1.8]{Tre08}.

  It now follows by straightforward elementary calculations for $2\times 2$ matrices that
  \[
   2\abs{\langle f,Wg\rangle} \le \norm{T_{f,g}} \le 2\max\{\langle f,V_0f\rangle,\langle g,V_1g\rangle\}
  \]
  and, therefore,
  \[
   2\abs{\langle f,Wg\rangle} \le \norm{V} \quad\text{ as well as }\quad \norm{T_{f,g}} \le 2\max\{\norm{V_0},\norm{V_1}\}.
  \]  
  Taking the supremum over all normalized $f$ and $g$ proves~\eqref{eq:help}.
 \end{proof}%
\end{lemma}

\begin{remark}
 Lemma~\ref{lem:semidefinite} is optimal in the sense that $\norm{W}=\norm{V}/2$ is possible and, at the same time,
 $\max\{\norm{V_0},\norm{V_1}\}$ may take any value between $\norm{V}/2$ and $\norm{V}$. This can be seen from the following
 example:

 Let $x>0$, $y\in[x/2,x]$, and consider the entries
 \[
  V_0 := \begin{pmatrix} y & 0\\ 0 & x/2 \end{pmatrix},\quad V_1 := \begin{pmatrix} x/2 & 0\\ 0 & y \end{pmatrix},\quad
  W := \begin{pmatrix} 0 & 0\\ x/2 & 0 \end{pmatrix}.
 \]
 Then, $V\ge 0$, $2\norm{W}=x=\norm{V}$, and $\norm{V_0}=\norm{V_1}=y$.
\end{remark}

\begin{proof}[Proof of Proposition~\ref{prop:sin2theta}]
 The case $\theta=\pi/2$ is obvious, so suppose that $\theta<\pi/2$. Recall from the proof of~\cite[Proposition~3.3]{Seel14} that
 in this case one has
 \[
  \sin2\theta \le \pi\frac{\norm{\EE_A(\Sigma)U^*VU\EE_A(\sigma)}}{d}
 \]
 with a certain unitary operator $U$ satisfying $U^*QU=\EE_A(\sigma)$. Also recall that the constant $\pi$ in this estimate can be
 replaced by $2$ if $\conv(\sigma)\cap\Sigma=\emptyset$ or $\sigma\cap\conv(\Sigma)=\emptyset$, see,
 e.g.,~\cite[Remark~3.2]{Seel14}. In order to complete the proof it only remains to observe that
 \[
  2\norm{\EE_A(\Sigma)U^*VU\EE_A(\sigma)} \le \norm{U^*VU} = \norm{V},
 \]
 where the inequality follows from Lemma~\ref{lem:semidefinite} by considering $U^*VU\ge0$ with respect to the orthogonal
 decomposition $\cH=\Ran\EE_A(\sigma)\oplus\Ran\EE_A(\Sigma)$.
\end{proof}%

\section*{Acknowledgements}
Parts of the material presented in this work are contained in the author's Ph.D.~thesis~\cite{SeelDiss}. The author is grateful to
Julian Gro{\ss}mann for helpful remarks on the manuscript.


\end{document}